\documentclass[twoside,reqno]{amsart}

\usepackage{latexsym}

\newcommand{\qdn}{\hspace*{-1.5mm}}
\newcommand{\qqdn}{\hspace*{-2.5mm}}
\newcommand{\xqdn}{\hspace*{-5.0mm}}
\newcommand{\xxqdn}{\hspace*{-10mm}}

\newcommand{\ffnk}[4]{\left[\qdn\ba{#1}#3\\#4\ea{\!\Big|\:#2}\right]}

\newcommand{\binm}{\binom}

\newcommand{\nnm}{\nonumber}
\newcommand{\be}{\begin{equation}}
\newcommand{\ee}{\end{equation}}
\newcommand{\ba}{\begin{array}}
\newcommand{\ea}{\end{array}}
\newcommand{\bmn}{\begin{eqnarray}}
\newcommand{\emn}{\end{eqnarray}}
\newcommand{\bnm}{\begin{eqnarray*}}
\newcommand{\enm}{\end{eqnarray*}}
\newcommand{\bln}{\begin{subequations}}
\newcommand{\eln}{\end{subequations}}

\newtheorem{thm}{Theorem}
\newtheorem{lemm}[thm]{Lemma}
\newtheorem{corl}[thm]{Corollary}

\newtheorem{entry}{Entry}

\newcommand{\bbtm}[4]{\bibitem{kn:#1}{#2,}~{#3,}~{#4.}}
\newcommand{\cito}[1]{\cite{kn:#1}}
\newcommand{\citu}[2]{\cite[#2]{kn:#1}}

\begin{document} 
{
\title{Saalsch\"{u}tz's theorem and summation formulae involving generalized harmonic numbers}
\author{Chuanan Wei}


\dedicatory{Department of Mathematics\\
  Shanghai Normal University, Shanghai 200234, China}

\thanks{\emph{Email address}: weichuanan78@ 163.com}


\begin{abstract}
In terms of the derivative operator, integral operator and
Saalsch\"{u}tz's theorem, two families of summation formulae
involving generalized harmonic numbers are established.
\end{abstract}

\maketitle\thispagestyle{empty}
\markboth{C. Wei}
         {Summation formulae involving generalized harmonic numbers}

\section{Introduction}
For a complex variable $x$, define the shifted factorial to be
\[(x)_{0}=0\quad \text{and}\quad (x)_{n}
=x(x+1)\cdots(x+n-1)\quad \text{with}\quad n\in\mathbb{N}.\]
 Following Andrews, Askey and Roy~\citu{andrews-r}{Chapter 2}, define the hypergeometric series by
\[\qdn_{1+r}F_s\ffnk{cccc}{z}{a_0,&a_1,&\cdots,&a_r}{&b_1,&\cdots,&b_s}
 \:=\:\sum_{k=0}^\infty\frac{(a_0)_k(a_1)_k\cdots(a_r)_k}{(1)_k(b_1)_k\cdots(b_s)_k}z^k,\]
where $\{a_{i}\}_{i\geq0}$ and $\{b_{j}\}_{j\geq1}$ are complex
parameters such that no zero factors appear in the denominators of
the summand on the right hand side. Then Saalsch\"{u}tz's theorem
(cf. \citu{andrews-r}{p. 69}) can be stated as
 \bmn\label{saalschutz}
 _3F_2\ffnk{cccc}{1}{a,b,-n}{c,1+a+b-c-n}
=\frac{(c-a)_n(c-b)_n}{(c)_n(c-a-b)_n}.
 \emn

For a complex number $x$ and a positive integer $\ell$, define
generalized harmonic numbers of $\ell$-order to be
\[H_{0}^{\langle \ell\rangle}(x)=0
\quad\text{and}\quad
 H_{n}^{\langle\ell\rangle}(x)=\sum_{k=1}^n\frac{1}{(x+k)^{\ell}}
 \quad\text{with}\quad n\in\mathbb{N}.\]
When $x=0$, they become harmonic numbers of $\ell$-order
\[H_{0}^{\langle \ell\rangle}=0
\quad\text{and}\quad
  H_{n}^{\langle \ell\rangle}
  =\sum_{k=1}^n\frac{1}{k^{\ell}} \quad\text{with}\quad n\in\mathbb{N}.\]
  Fixing $\ell=1$ in $H_{0}^{\langle \ell\rangle}(x)$ and $H_{n}^{\langle \ell\rangle}(x)$, we obtain
generalized harmonic numbers
\[H_{0}(x)=0
\quad\text{and}\quad H_{n}(x)
  =\sum_{k=1}^n\frac{1}{x+k} \quad\text{with}\quad n\in\mathbb{N}.\]
When $x=0$, they reduce to classical harmonic numbers
\[H_{0}=0\quad\text{and}\quad
H_{n}=\sum_{k=1}^n\frac{1}{k} \quad\text{with}\quad
n\in\mathbb{N}.\]

 For a differentiable function $f(x)$, define the derivative operator
$\mathcal{D}_x$ by
 \bnm
\mathcal{D}_xf(x)=\frac{d}{dx}f(x).
 \enm
For an integrable function $g(x)$, define the integral operator
$\mathcal{I}_x$ by
 \bnm
\mathcal{I}_xg(x)=\int_0^{x}g(x)dx.
 \enm
In order to explain the relation of the derivative operator and
generalized harmonic numbers, we introduce the following lemma.

\begin{lemm} \label{lemm-a}
 Let $x$ and  $\{a_j,b_j,c_j,d_j\}_{j=1}^s$ be all complex
numbers. Then
 \bnm
\mathcal{D}_x\prod_{j=1}^s\frac{a_jx+b_j}{c_jx+d_j}=\prod_{j=1}^s\frac{a_jx+b_j}{c_jx+d_j}
 \sum_{j=1}^s\frac{a_jd_j-b_jc_j}{(a_jx+b_j)(c_jx+d_j)}.
 \enm
\end{lemm}

\begin{proof}
It is not difficult to verify the case $s=1$ of Lemma \ref{lemm-a}.
Suppose that
 \bnm
\mathcal{D}_x\prod_{j=1}^m\frac{a_jx+b_j}{c_jx+d_j}=\prod_{j=1}^m\frac{a_jx+b_j}{c_jx+d_j}
 \sum_{j=1}^m\frac{a_jd_j-b_jc_j}{(a_jx+b_j)(c_jx+d_j)}
 \enm
is true. We can proceed as follows:
 \bnm
&&\xqdn\mathcal{D}_x\prod_{j=1}^{m+1}\frac{a_jx+b_j}{c_jx+d_j}=
 \mathcal{D}_x\bigg\{\prod_{j=1}^{m}\frac{a_jx+b_j}{c_jx+d_j}\frac{a_{m+1}x+b_{m+1}}{c_{m+1}x+d_{m+1}}\bigg\}\\
&&\xqdn\:=\:\frac{a_{m+1}x+b_{m+1}}{c_{m+1}x+d_{m+1}}\mathcal{D}_x\prod_{j=1}^{m}\frac{a_jx+b_j}{c_jx+d_j}
+\prod_{j=1}^{m}\frac{a_jx+b_j}{c_jx+d_j}\mathcal{D}_x\frac{a_{m+1}x+b_{m+1}}{c_{m+1}x+d_{m+1}}\\
&&\xqdn\:=\:\frac{a_{m+1}x+b_{m+1}}{c_{m+1}x+d_{m+1}}\prod_{j=1}^m\frac{a_jx+b_j}{c_jx+d_j}\sum_{j=1}^m\frac{a_jd_j-b_jc_j}{(a_jx+b_j)(c_jx+d_j)}\\
&&\xqdn\:+\:\prod_{j=1}^{m}\frac{a_jx+b_j}{c_jx+d_j}\frac{a_{m+1}d_{m+1}-b_{m+1}c_{m+1}}{(c_{m+1}x+d_{m+1})^2}\\
&&\xqdn\:=\:\prod_{j=1}^{m+1}\frac{a_jx+b_j}{c_jx+d_j}
 \bigg\{\sum_{j=1}^m\frac{a_jd_j-b_jc_j}{(a_jx+b_j)(c_jx+d_j)}+\frac{a_{m+1}d_{m+1}-b_{m+1}c_{m+1}}{(a_{m+1}x+b_{m+1})(c_{m+1}x+d_{m+1})}\bigg\}\\
&&\xqdn\:=\:\prod_{j=1}^{m+1}\frac{a_jx+b_j}{c_jx+d_j}
 \sum_{j=1}^{m+1}\frac{a_jd_j-b_jc_j}{(a_jx+b_j)(c_jx+d_j)}.
 \enm
This proves Lemma \ref{lemm-a} inductively.
\end{proof}

Setting $a_j=1,b_j=r-j+1,c_j=0,d_j=j$ in Lemma \ref{lemm-a}, it is
easy to find that
$$\mathcal{D}_x\:\binm{x+r}{s}=\binm{x+r}{s}\big\{H_r(x)-H_{r-s}(x)\big\},$$
where $r,s\in\mathbb{N}_0$ with $s\leq r$. Besides, we have the
following relation:
$$\mathcal{D}_xH_{n}^{\langle \ell\rangle}(x)=-\ell H_{n}^{\langle
\ell+1\rangle}(x).$$

As pointed out by Richard Askey (cf. \cito{andrews}), expressing
harmonic numbers in accordance with differentiation of binomial
coefficients can be traced back to Issac Newton.
 In 2003, Paule and Schneider
\cito{paule} computed the family of series:
 \bnm
\quad
W_n(\alpha)=\sum_{k=0}^n\binm{n}{k}^{\alpha}\{1+\alpha(n-2k)H_k\}
 \enm
with $\alpha=1,2,3,4,5$ by combining this way with Zeilberger's
algorithm for definite hypergeometric sums. According to the
derivative operator and the hypergeometric form of Andrews'
$q$-series transformation, Krattenthaler and Rivoal
\cito{krattenthaler} deduced general Paule-Schneider type identities
with $\alpha$ being a positive integer.  More results from
differentiation of binomial coefficients can be seen in the papers
\cite{kn:sofo-a,kn:wang-b,kn:wei-a,kn:wei-b}. For different ways and
related harmonic number identities, the reader may refer to
\cite{kn:chen,kn:kronenburg-a,
kn:kronenburg-b,kn:schneider,kn:sofo-b,kn:wang-a}. It should be
mentioned that Sun \cito{sun} showed recently some congruence
relations concerning harmonic numbers to us.

Inspired by the work just mentioned, we shall explore, by means of
the derivative operator, integral operator and \eqref{saalschutz},
closed expressions for the following two families of series:
 \bnm
 &&\:\sum_{k=0}^{n}(-1)^k\binm{n}{k}\frac{\binm{2x-y+n+k}{k}\binm{y+k}{k}}{\binm{x+k}{k}^2}
 \frac{\binm{y}{t}}{\binm{y+k}{t}}H_{k}^{\langle2\rangle}(x),\\
 &&\:\sum_{k=0}^{n}(-1)^k\binm{n}{k}\frac{\binm{y+k}{k}}{\binm{y-n+k}{k}}
 \frac{\binm{y}{t}}{\binm{y+k}{t}}H_{k}^{\langle\ell\rangle}(x),
 \enm
where $t\in\mathbb{N}$. In order to avoid appearance of complicated
expressions, our explicit formulae are offered only for $t=1,2$ and
$\ell=1,2,3,4$.

\section{The first family of summation formulae involving\\ generalized harmonic numbers}
\begin{thm} \label{thm-a}
Let $x$ and $y$ be both complex numbers. Then
 \bnm
&&\sum_{k=0}^n(-1)^k\binm{n}{k}\frac{\binm{2x-y+n+k}{k}\binm{y+k}{k}}{\binm{x+k}{k}^2}
\frac{y}{y+k}H_{k}^{\langle2\rangle}(x)\\
&&\:\,=\frac{\binm{x-y+n}{n}^2}{\binm{x+n}{n}^2}\Big\{H_{n}^{\langle2\rangle}(x)-H_{n}^{\langle2\rangle}(x-y)\Big\}.
 \enm
\end{thm}

\begin{proof}
Perform the replacements $a\to1+z$, $b\to y$, $c\to1+x$ in
\eqref{saalschutz} to get
 \bmn\label{saalschutz-a}
\qquad\quad\sum_{k=0}^n(-1)^k\binm{n}{k}\frac{\binm{z+k}{k}\binm{y+k}{k}}{\binm{x+k}{k}\binm{y+z-x-n+k}{k}}\frac{y}{y+k}
=\frac{\binm{x-y+n}{n}\binm{x-z-1+n}{n}}{\binm{x+n}{n}\binm{x-y-z-1+n}{n}}.
 \emn
Applying the derivative operator $\mathcal{D}_x$ to both sides of
\eqref{saalschutz-a}, we gain
 \bnm
&&\xqdn\qdn\sum_{k=0}^n(-1)^k\binm{n}{k}\frac{\binm{z+k}{k}\binm{y+k}{k}}{\binm{x+k}{k}\binm{y+z-x-n+k}{k}}\frac{y}{y+k}
\Big\{H_k(y+z-x-n)-H_k(x)\Big\}\\
&&\xqdn\qdn\,\,=\,\frac{\binm{x-y+n}{n}\binm{x-z-1+n}{n}}{\binm{x+n}{n}\binm{x-y-z-1+n}{n}}
\Big\{H_n(x-y)+H_n(x-z-1)-H_n(x)-H_n(x-y-z-1)\Big\}.
 \enm
The equivalent form of it reads as
 \bmn\label{harmonic-a}
&&\xxqdn\sum_{k=0}^n(-1)^k\binm{n}{k}\frac{\binm{z+k}{k}\binm{y+k}{k}}{\binm{x+k}{k}\binm{y+z-x-n+k}{k}}\frac{y}{y+k}
\sum_{i=1}^k\frac{1}{(x+i)(y+z-x-n+i)}
 \nnm\\
&&\xxqdn\,\,=\,\frac{\binm{x-y+n}{n}\binm{x-z-1+n}{n}}{\binm{x+n}{n}\binm{x-y-z-1+n}{n}}
\bigg\{\frac{H_n(x-y)\!+\!H_n(x-z-1)}{2x-y-z+n}\!-\!\frac{H_n(x)\!+\!H_n(x-y-z-1)}{2x-y-z+n}\bigg\}.
 \emn
By means of L'H\^{o}spital rule, we achieve
 \bmn\label{limit-a}
&&\xqdn\text{Lim}_{z\to2x-y+n}\frac{H_n(x-y)+H_n(x-z-1)}{2x-y-z+n}
 \nnm\\\nnm
&&\xqdn\:\:=\:\text{Lim}_{z\to2x-y+n}\frac{H_n^{\langle2\rangle}(x-z-1)}{-1}\\
&&\xqdn\:\:=\:-H_n^{\langle2\rangle}(y-x-n-1)
 \nnm\\
&&\xqdn\:\:=\:-H_n^{\langle2\rangle}(x-y),
 \emn
 \bmn\label{limit-b}
&&\xqdn\text{Lim}_{z\to2x-y+n}\frac{H_n(x)+H_n(x-y-z-1)}{2x-y-z+n}
 \nnm\\\nnm
&&\xqdn\:\:=\:\text{Lim}_{z\to2x-y+n}\frac{H_n^{\langle2\rangle}(x-y-z-1)}{-1}\\
&&\xqdn\:\:=\:-H_n^{\langle2\rangle}(-x-n-1)
 \nnm\\
&&\xqdn\:\:=\:-H_n^{\langle2\rangle}(x).
 \emn
Taking the limit $z\to2x-y+n$ on both sides of \eqref{harmonic-a}
and using \eqref{limit-a}-\eqref{limit-b}, we attain Theorem
\ref{thm-a} to complete the proof.
\end{proof}

Choosing $x=p$, $y=q$ in  Theorem \ref{thm-a} with
$p,q\in\mathbb{N}_0$ and utilizing \eqref{saalschutz-a}, we obtain
the summation formula involving harmonic numbers of 2-order.

\begin{corl} \label{corl-a}
Let $p$ and $q$ be both nonnegative integers satisfying $p\geq q$.
Then
 \bnm
&&\qdn\qqdn\sum_{k=0}^n(-1)^k\binm{n}{k}\frac{\binm{2p-q+n+k}{k}\binm{q+k}{k}}{\binm{p+k}{k}^2}\frac{q}{q+k}H_{p+k}^{\langle2\rangle}\\
&&\qdn\qqdn\:\,=\frac{\binm{p-q+n}{n}^2}{\binm{p+n}{n}^2}\Big\{H_{p-q}^{\langle2\rangle}+H_{p+n}^{\langle2\rangle}-H_{p-q+n}^{\langle2\rangle}\Big\}.
 \enm
\end{corl}

\begin{thm} \label{thm-b}
Let $x$ and $y$ be both complex numbers. Then
 \bnm
&&\xqdn\sum_{k=0}^n(-1)^k\binm{n}{k}\frac{\binm{2x-y+n+k}{k}\binm{y+k}{k}}{\binm{x+k}{k}^2}
\frac{(y-1)y}{(y+k-1)(y+k)}H_{k}^{\langle2\rangle}(x)\\
&&\xqdn\:\,=\frac{n^2+n(1+2x-y)+(1+x-y)^2}{(1+x-y)^2}\frac{\binm{x-y+n}{n}^2}{\binm{x+n}{n}^2}
\Big\{H_{n}^{\langle2\rangle}(x)-H_{n}^{\langle2\rangle}(x-y)\Big\}\\
&&\xqdn\:\,+\:\frac{n^2+n(1+2x-y)}{(1+x-y)^4}\frac{\binm{x-y+n}{n}^2}{\binm{x+n}{n}^2}.
 \enm
\end{thm}

\begin{proof}
Replace $c$ by $1+c$ in \eqref{saalschutz} to get
 \bnm\qquad\qquad
 _3F_2\ffnk{cccc}{1}{a,b,-n}{1+c,a+b-c-n}
=\frac{(1+c-a)_n(1+c-b)_n}{(1+c)_n(1+c-a-b)_n}.
 \enm
The combination of \eqref{saalschutz} and the last equation gives
\bmn\label{saalschutz-b}
 \quad_3F_2\ffnk{cccc}{1}{a,b,-n}{1+c,1+a+b-c-n}
=\bigg\{1+\frac{n(c-a-b)}{(c-a)(c-b)}\bigg\}\frac{(c-a)_n(c-b)_n}{(1+c)_n(c-a-b)_n}.
 \emn
Employ the substitutions $a\to1+z$, $b\to y-1$, $c\to x$ in
\eqref{saalschutz-b} to gain
 \bmn\label{saalschutz-c}
&&\qqdn\xxqdn\sum_{k=0}^n(-1)^k\binm{n}{k}\frac{\binm{z+k}{k}\binm{y+k}{k}}{\binm{x+k}{k}\binm{y+z-x-n+k}{k}}\frac{(y-1)y}{(y+k-1)(y+k)}
 \nnm\\
&&\qqdn\xxqdn\,\,=\,\frac{(x-y+1)(x-z-1)+n(x-y-z)}{(x-y+1)(x-z-1+n)}
\frac{\binm{x-y+n}{n}\binm{x-z-1+n}{n}}{\binm{x+n}{n}\binm{x-y-z-1+n}{n}}.
 \emn
Applying the derivative operator $\mathcal{D}_x$ to both sides of
\eqref{saalschutz-c}, we achieve
 \bnm
&&\xqdn\sum_{k=0}^n(-1)^k\binm{n}{k}\frac{\binm{z+k}{k}\binm{y+k}{k}}{\binm{x+k}{k}\binm{y+z-x-n+k}{k}}
\frac{(y-1)y}{(y+k-1)(y+k)}\Big\{H_k(y+z-x-n)-H_k(x)\Big\}\\
&&\xqdn\,\,=\,\frac{(x-y+1)(x-z-1)+n(x-y-z)}{(x-y+1)(x-z-1+n)}
\frac{\binm{x-y+n}{n}\binm{x-z-1+n}{n}}{\binm{x+n}{n}\binm{x-y-z-1+n}{n}}\\
&&\xqdn\,\,\times\:\:\Big\{H_n(x-y)+H_n(x-z-1)-H_n(x)-H_n(x-y-z-1)\Big\}\\
&&\xqdn\,\,+\,\,\frac{n(z+1)(2x-y-z+n)}{(x-y+1)^2(x-z-1+n)^2}
\frac{\binm{x-y+n}{n}\binm{x-z-1+n}{n}}{\binm{x+n}{n}\binm{x-y-z-1+n}{n}}.
 \enm
The equivalent form of it can be expressed as
 \bnm
&&\xqdn\sum_{k=0}^n(-1)^k\binm{n}{k}\frac{\binm{z+k}{k}\binm{y+k}{k}}{\binm{x+k}{k}\binm{y+z-x-n+k}{k}}\frac{(y-1)y}{(y+k-1)(y+k)}
\sum_{i=1}^k\frac{1}{(x+i)(y+z-x-n+i)}\\
&&\xqdn\,\,=\,\frac{(x-y+1)(x-z-1)+n(x-y-z)}{(x-y+1)(x-z-1+n)}
\frac{\binm{x-y+n}{n}\binm{x-z-1+n}{n}}{\binm{x+n}{n}\binm{x-y-z-1+n}{n}}\\
&&\xqdn\,\,\times\:\:\bigg\{\frac{H_n(x-y)\!+\!H_n(x-z-1)}{2x-y-z+n}\!-\!\frac{H_n(x)\!+\!H_n(x-y-z-1)}{2x-y-z+n}\bigg\}\\
&&\xqdn\,\,+\,\,\frac{n(z+1)}{(x-y+1)^2(x-z-1+n)^2}
\frac{\binm{x-y+n}{n}\binm{x-z-1+n}{n}}{\binm{x+n}{n}\binm{x-y-z-1+n}{n}}.
 \enm
Taking the limit $z\to2x-y+n$ on both sides of the last equation and
exploiting \eqref{limit-a}-\eqref{limit-b}, we attain Theorem
\ref{thm-b} to finish the proof.
\end{proof}

Selecting $x=p$, $y=q$ in  Theorem \ref{thm-b} with
$p,q\in\mathbb{N}_0$ and availing \eqref{saalschutz-c}, we obtain
the summation formula involving harmonic numbers of 2-order.

\begin{corl} \label{corl-b}
Let $p$ and $q$ be both nonnegative integers provided that $p\geq
q$. Then
 \bnm
&&\xqdn\sum_{k=0}^n(-1)^k\binm{n}{k}\frac{\binm{2p-q+n+k}{k}\binm{q+k}{k}}{\binm{p+k}{k}^2}\frac{(q-1)q}{(q+k-1)(q+k)}H_{p+k}^{\langle2\rangle}\\
&&\xqdn\:\,=\frac{n^2+n(1+2p-q)+(1+p-q)^2}{(1+p-q)^2}\frac{\binm{p-q+n}{n}^2}{\binm{p+n}{n}^2}
\Big\{H_{p-q}^{\langle2\rangle}+H_{p+n}^{\langle2\rangle}-H_{p-q+n}^{\langle2\rangle}\Big\}\\
&&\xqdn\:\,+\:\frac{n^2+n(1+2p-q)}{(1+p-q)^4}\frac{\binm{p-q+n}{n}^2}{\binm{p+n}{n}^2}.
 \enm
\end{corl}

Similarly, closed expressions for the following series
 \bnm\qquad
 \sum_{k=0}^{n}(-1)^k\binm{n}{k}\frac{\binm{2x-y+n+k}{k}\binm{y+k}{k}}{\binm{x+k}{k}^2}
 \frac{\binm{y}{t}}{\binm{y+k}{t}}H_{k}^{\langle2\rangle}(x)
 \enm
with $t\geq2$ can also be derived. The corresponding results will
not be displayed here.
\section{The second family of summation formulae involving\\ generalized harmonic numbers}
\begin{thm} \label{thm-c}
Let $x$ and $y$ be both complex numbers. Then
 \bnm
&&\xxqdn\qdn\sum_{k=0}^n(-1)^k\binm{n}{k}\frac{\binm{y+k}{k}}{\binm{y-n+k}{k}}
\frac{y}{y+k}H_{k}^{\langle2\rangle}(x)\\
&&\xxqdn\qdn\:\,=\frac{(-1)^n}{n}\frac{\binm{x-y+n}{n}}{\binm{x+n}{n}\binm{y}{n}}\Big\{H_{n}(x-y)-H_{n}(x)\Big\}.
 \enm
\end{thm}

\begin{proof}
Perform the replacements $a\to1+x$, $b\to y$, $c\to1+z$ in
\eqref{saalschutz} to get
 \bmn\label{saalschutz-d}
\qqdn\sum_{k=0}^n(-1)^k\binm{n}{k}\frac{\binm{x+k}{k}\binm{y+k}{k}}{\binm{z+k}{k}\binm{x+y-z-n+k}{k}}\frac{y}{y+k}
=\frac{\binm{z-x-1+n}{n}\binm{z-y+n}{n}}{\binm{z-x-y-1+n}{n}\binm{z+n}{n}}.
 \emn
Applying the derivative operator $\mathcal{D}_x$ to both sides of
\eqref{saalschutz-d}, we have
 \bnm
&&\xqdn\sum_{k=0}^n(-1)^k\binm{n}{k}\frac{\binm{x+k}{k}\binm{y+k}{k}}{\binm{z+k}{k}\binm{x+y-z-n+k}{k}}\frac{y}{y+k}
\Big\{H_k(x)-H_k(x+y-z-n)\Big\}\\
&&\xqdn\,\,=\,\frac{\binm{z-x-1+n}{n}\binm{z-y+n}{n}}{\binm{z-x-y-1+n}{n}\binm{z+n}{n}}
\Big\{H_n(z-x-y-1)-H_n(z-x-1)\Big\}.
 \enm
The equivalent form of it reads as
 \bnm
&&\xqdn\sum_{k=0}^n(-1)^k\binm{n}{k}\frac{\binm{x+k}{k}\binm{y+k}{k}}{\binm{z+k}{k}\binm{x+y-z-n+k}{k}}\frac{y}{y+k}
\sum_{i=1}^k\frac{1}{(x+i)(x+y-z-n+i)}\\
&&\xqdn\,\,=\,\frac{\binm{z-x-1+n}{n}\binm{z-y-1+n}{n}}{\binm{z-x-y-1+n}{n}\binm{z+n}{n}}
\frac{H_n(z-x-y-1)-H_n(z-x-1)}{y-z}.
 \enm
 Taking the limit $z\to y-n$ on both sides of the last equation, we gain Theorem \ref{thm-c} to
complete the proof.
\end{proof}

Fixing $x=p$, $y=q$ in  Theorem \ref{thm-c} with
$p,q\in\mathbb{N}_0$ and using \eqref{saalschutz-d}, we achieve the
summation formula involving harmonic numbers of 2-order.

\begin{corl} \label{corl-c}
Let $p$ and $q$ be both nonnegative integers satisfying $p\geq q\geq
n$. Then
 \bnm
&&\xxqdn\sum_{k=0}^n(-1)^k\binm{n}{k}\frac{\binm{q+k}{k}}{\binm{q-n+k}{k}}\frac{q}{q+k}H_{p+k}^{\langle2\rangle}\\
&&\xxqdn\:\,=\frac{(-1)^n}{n}\frac{\binm{p-q+n}{n}}{\binm{p+n}{n}\binm{q}{n}}\Big\{H_{p-q+n}-H_{p+n}-H_{p-q}+H_p\Big\}.
 \enm
\end{corl}

\begin{thm} \label{thm-d}
Let $x$ and $y$ be both complex numbers. Then
 \bnm
\quad\sum_{k=0}^n(-1)^k\binm{n}{k}\frac{\binm{y+k}{k}}{\binm{y-n+k}{k}}
\frac{y}{y+k}H_{k}(x)
=\frac{(-1)^n}{n}\frac{1}{\binm{y}{n}}\bigg\{1-\frac{\binm{x-y+n}{n}}{\binm{x+n}{n}}\bigg\}.
 \enm
\end{thm}

\begin{proof}
Applying the integral operator $\mathcal{I}_x$ to both sides of
Theorem \ref{thm-c}, we attain
  \bmn\label{harmonic-b}
&&\xxqdn\sum_{k=0}^n(-1)^k\binm{n}{k}\frac{\binm{y+k}{k}}{\binm{y-n+k}{k}}
\frac{y}{y+k}\Big\{H_k-H_{k}(x)\Big\}
 \nnm\\\nnm
&&\xxqdn\:\,=\frac{(-1)^n}{n}\frac{\binm{x-y+n}{n}}{\binm{x+n}{n}\binm{y}{n}}\bigg|_0^x\\
&&\xxqdn\:\,=\frac{(-1)^n}{n}\frac{\binm{x-y+n}{n}}{\binm{x+n}{n}\binm{y}{n}}-
\frac{(-1)^n}{n}\frac{\binm{-y+n}{n}}{\binm{y}{n}}.
 \emn
Take the limit $x\to\infty$ on both sides of \eqref{harmonic-b} to
deduce
 \bnm
\sum_{k=0}^n(-1)^k\binm{n}{k}\frac{\binm{y+k}{k}}{\binm{y-n+k}{k}}
\frac{y}{y+k}H_k
 =\frac{(-1)^n}{n}\frac{1}{\binm{y}{n}}-
\frac{(-1)^n}{n}\frac{\binm{-y+n}{n}}{\binm{y}{n}}.
 \enm
The difference of \eqref{harmonic-b} and the last equation creates
Theorem \ref{thm-d}.
\end{proof}

Setting $x=p$, $y=q$ in  Theorem \ref{thm-d} with
$p,q\in\mathbb{N}_0$ and utilizing \eqref{saalschutz-d}, we obtain
the summation formula involving  harmonic numbers.

\begin{corl} \label{corl-d}
Let $p$ and $q$ be both nonnegative integers provided that $q\geq
n$. Then
 \bnm
\:\sum_{k=0}^n(-1)^k\binm{n}{k}\frac{\binm{q+k}{k}}{\binm{q-n+k}{k}}\frac{q}{q+k}H_{p+k}
=\frac{(-1)^n}{n}\frac{1}{\binm{q}{n}}\bigg\{1-\frac{\binm{p-q+n}{n}}{\binm{p+n}{n}}\bigg\}.
 \enm
\end{corl}

Applying the derivative operator $\mathcal{D}_x$ to both sides of
Theorem \ref{thm-d}, we get the summation formula involving
generalized harmonic numbers of 3-order.

\begin{thm} \label{thm-e}
Let $x$ and $y$ be both complex numbers. Then
 \bnm
&&\xqdn\sum_{k=0}^n(-1)^k\binm{n}{k}\frac{\binm{y+k}{k}}{\binm{y-n+k}{k}}
\frac{y}{y+k}H_{k}^{\langle3\rangle}(x)=\frac{(-1)^n}{2n}\frac{\binm{x-y+n}{n}}{\binm{x+n}{n}\binm{y}{n}}\\
&&\xqdn\:\,\times\:
\Big\{\big[H_{n}^{\langle2\rangle}(x-y)-H_{n}^{\langle2\rangle}(x)\big]-\big[H_{n}(x-y)-H_{n}(x)\big]^2\Big\}.
 \enm
\end{thm}

Choosing $x=p$, $y=q$ in  Theorem \ref{thm-e} with
$p,q\in\mathbb{N}_0$ and exploiting \eqref{saalschutz-d}, we gain
the summation formula involving harmonic numbers of 3-order.

\begin{corl} \label{corl-e}
Let $p$ and $q$ be both nonnegative integers satisfying $p\geq q\geq
n$. Then
 \bnm
&&\xxqdn\sum_{k=0}^n(-1)^k\binm{n}{k}\frac{\binm{q+k}{k}}{\binm{q-n+k}{k}}\frac{q}{q+k}H_{p+k}^{\langle3\rangle}
=\frac{(-1)^n}{2n}\frac{\binm{p-q+n}{n}}{\binm{p+n}{n}\binm{q}{n}}
\\&&\xxqdn\:\,\times\:
\Big\{\big[H_{p-q+n}^{\langle2\rangle}-H_{p+n}^{\langle2\rangle}-H_{p-q}^{\langle2\rangle}+H_p^{\langle2\rangle}\big]
  -\big[H_{p-q+n}-H_{p+n}-H_{p-q}+H_p\big]^2\Big\}.
 \enm
\end{corl}

Applying the derivative operator $\mathcal{D}_x$ to both sides of
Theorem \ref{thm-e}, we achieve the summation formula involving
generalized harmonic numbers of 4-order.

\begin{thm} \label{thm-f}
Let $x$ and $y$ be both complex numbers. Then
 \bnm
&&\xqdn\sum_{k=0}^n(-1)^k\binm{n}{k}\frac{\binm{y+k}{k}}{\binm{y-n+k}{k}}
\frac{y}{y+k}H_{k}^{\langle4\rangle}(x)=\frac{(-1)^n}{6n}\frac{\binm{x-y+n}{n}}{\binm{x+n}{n}\binm{y}{n}}\\
&&\xqdn\:\,\times\:
\Big\{\big[H_{n}(x-y)-H_{n}(x)\big]^3+2\big[H_{n}^{\langle3\rangle}(x-y)-H_{n}^{\langle3\rangle}(x)\big]\\
&&\quad-3\big[H_{n}(x-y)-H_{n}(x)\big]\big[H_{n}^{\langle2\rangle}(x-y)-H_{n}^{\langle2\rangle}(x)\big]\:\Big\}.
 \enm
\end{thm}

Selecting $x=p$, $y=q$ in  Theorem \ref{thm-f} with
$p,q\in\mathbb{N}_0$ and availing \eqref{saalschutz-d}, we attain
the summation formula involving harmonic numbers of 4-order.

\begin{corl} \label{corl-f}
Let $p$ and $q$ be both nonnegative integers provided that $p\geq
q\geq n$. Then
 \bnm
&&\xxqdn\qdn\sum_{k=0}^n(-1)^k\binm{n}{k}\frac{\binm{q+k}{k}}{\binm{q-n+k}{k}}\frac{q}{q+k}H_{p+k}^{\langle4\rangle}
=\frac{(-1)^n}{6n}\frac{\binm{p-q+n}{n}}{\binm{p+n}{n}\binm{q}{n}}
\\&&\xxqdn\qdn\:\,\times\:
\Big\{\big[H_{p-q+n}-H_{p+n}-H_{p-q}+H_p\big]^3+
2\big[H_{p-q+n}^{\langle3\rangle}-H_{p+n}^{\langle3\rangle}-H_{p-q}^{\langle3\rangle}+H_p^{\langle3\rangle}\big]\\
&&\qdn-3\big[H_{p-q+n}-H_{p+n}-H_{p-q}+H_p\big]
  \big[H_{p-q+n}^{\langle2\rangle}-H_{p+n}^{\langle2\rangle}-H_{p-q}^{\langle2\rangle}+H_p^{\langle2\rangle}\big]\Big\}.
 \enm
\end{corl}

\begin{thm} \label{thm-g}
Let $x$ and $y$ be both complex numbers. Then
 \bnm
&&\xqdn\sum_{k=0}^n(-1)^k\binm{n}{k}\frac{\binm{y+k}{k}}{\binm{y-n+k}{k}}
\frac{(y-1)y}{(y+k-1)(y+k)}H_{k}^{\langle2\rangle}(x)\\
&&\xqdn\:\,=\frac{(-1)^n(1+x-y+ny)}{n(n-1)(1+x-y)}\frac{\binm{x-y+n}{n}}{\binm{x+n}{n}\binm{y}{n}}\\
 &&\xqdn\:\,\times\:\bigg\{H_{n}(x)-H_{n}(x-y)+\frac{ny}{(1+x-y)(1+x-y+ny)}\bigg\}.
 \enm
\end{thm}

\begin{proof}
Employ the substitutions $a\to1+x$, $b\to y-1$, $c\to z$ in
\eqref{saalschutz-b} to obtain
  \bmn\label{saalschutz-e}
&&\sum_{k=0}^n(-1)^k\binm{n}{k}\frac{\binm{x+k}{k}\binm{y+k}{k}}{\binm{z+k}{k}\binm{x+y-z-n+k}{k}}\frac{(y-1)y}{(y+k-1)(y+k)}
 \nnm\\
&&\,\,=\,\frac{(z-x-1)(z-y+1)+n(z-x-y)}{(z-x-1+n)(z-y+1)}
\frac{\binm{z-x-1+n}{n}\binm{z-y+n}{n}}{\binm{z-x-y-1+n}{n}\binm{z+n}{n}}.
 \emn
Applying the derivative operator $\mathcal{D}_x$ to both sides of
\eqref{saalschutz-e}, we get
 \bnm
&&\xqdn\sum_{k=0}^n(-1)^k\binm{n}{k}\frac{\binm{x+k}{k}\binm{y+k}{k}}{\binm{z+k}{k}\binm{x+y-z-n+k}{k}}
\frac{(y-1)y}{(y+k-1)(y+k)}\Big\{H_k(x)-H_k(x+y-z-n)\Big\}\\
&&\xqdn\,\,=\,\frac{(z-x-1)(z-y+1)+n(z-x-y)}{(z-x-1+n)(z-y+1)}
\frac{\binm{z-x-1+n}{n}\binm{z-y+n}{n}}{\binm{z-x-y-1+n}{n}\binm{z+n}{n}}\\
&&\xqdn\,\,\times\:\Big\{H_n(z-x-y-1)-H_n(z-x-1)\Big\}\\
&&\xqdn\,\,-\:\frac{n(z+n)}{(z-x-1+n)^2(z-y+1)}
\frac{\binm{z-x-1+n}{n}\binm{z-y+n}{n}}{\binm{z-x-y-1+n}{n}\binm{z+n}{n}}.
 \enm
Its equivalent form can be written as
 \bnm
&&\xqdn\sum_{k=0}^n(-1)^k\binm{n}{k}\frac{\binm{x+k}{k}\binm{y+k}{k}}{\binm{z+k}{k}\binm{x+y-z-n+k}{k}}
\frac{(y-1)y}{(y+k-1)(y+k)}\sum_{i=1}^k\frac{1}{(x+i)(x+y-z-n+i)}\\
&&\xqdn\,\,=\,\frac{(z-x-1)(z-y+1)+n(z-x-y)}{(z-x-1+n)(z-y+1)(y-z)}
\frac{\binm{z-x-1+n}{n}\binm{z-y-1+n}{n}}{\binm{z-x-y-1+n}{n}\binm{z+n}{n}}\\
&&\xqdn\,\,\times\:\Big\{H_n(z-x-y-1)-H_n(z-x-1)\Big\}\\
&&\xqdn\,\,-\:\frac{n(z+n)}{(z-x-1+n)^2(z-y+1)(y-z)}
\frac{\binm{z-x-1+n}{n}\binm{z-y-1+n}{n}}{\binm{z-x-y-1+n}{n}\binm{z+n}{n}}.
 \enm
 Taking the limit $z\to y-n$ on both sides of the last equation, we gain Theorem \ref{thm-g} to
finish the proof.
\end{proof}

Fixing $x=p$, $y=q$ in  Theorem \ref{thm-g} with
$p,q\in\mathbb{N}_0$ and using \eqref{saalschutz-e}, we achieve the
summation formula involving harmonic numbers of 2-order.

\begin{corl} \label{corl-g}
Let $p$ and $q$ be both nonnegative integers satisfying $p\geq q\geq
n$. Then
 \bnm
&&\qqdn\xqdn\sum_{k=0}^n(-1)^k\binm{n}{k}\frac{\binm{q+k}{k}}{\binm{q-n+k}{k}}
\frac{(q-1)q}{(q+k-1)(q+k)}H_{p+k}^{\langle2\rangle}\\
&&\qqdn\xqdn\:\,=\frac{(-1)^n(1+p-q+nq)}{n(n-1)(1+p-q)}\frac{\binm{p-q+n}{n}}{\binm{p+n}{n}\binm{q}{n}}\\
 &&\qqdn\xqdn\:\,\times\:\bigg\{H_{p+n}-H_{p-q+n}-H_{p}+H_{p-q}+\frac{nq}{(1+p-q)(1+p-q+nq)}\bigg\}.
 \enm
\end{corl}

\begin{thm} \label{thm-h}
Let $x$ and $y$ be both complex numbers. Then
 \bnm
&&\xqdn\sum_{k=0}^n(-1)^k\binm{n}{k}\frac{\binm{y+k}{k}}{\binm{y-n+k}{k}}
\frac{(y-1)y}{(y+k-1)(y+k)}H_{k}(x)\\
&&\xqdn\:\,=\frac{(-1)^n(1+x-y+ny)}{n(n-1)(1+x-y)}\frac{\binm{x-y+n}{n}}{\binm{x+n}{n}\binm{y}{n}}
 -\frac{(-1)^n}{n(n-1)}\frac{1}{\binm{y}{n}}.
 \enm
\end{thm}

\begin{proof}
Applying the integral operator $\mathcal{I}_x$ to both sides of
Theorem \ref{thm-g}, we attain
  \bmn\label{harmonic-c}
&&\xxqdn\sum_{k=0}^n(-1)^k\binm{n}{k}\frac{\binm{y+k}{k}}{\binm{y-n+k}{k}}
\frac{(y-1)y}{(y+k-1)(y+k)}\Big\{H_k-H_{k}(x)\Big\}
 \nnm\\\nnm
&&\xxqdn\:\,=\frac{(-1)^{n+1}(1+x-y+ny)}{n(n-1)(1+x-y)}\frac{\binm{x-y+n}{n}}{\binm{x+n}{n}\binm{y}{n}}\bigg|_0^x\\
&&\xxqdn\:\,=\frac{(-1)^n(1-y+ny)}{n(n-1)(1-y)}\frac{\binm{-y+n}{n}}{\binm{y}{n}}
-\frac{(-1)^n(1+x-y+ny)}{n(n-1)(1+x-y)}\frac{\binm{x-y+n}{n}}{\binm{x+n}{n}\binm{y}{n}}.
 \emn
Take the limit $x\to\infty$ on both sides of \eqref{harmonic-c} to
derive
 \bnm
&&\qdn\qqdn\xxqdn\sum_{k=0}^n(-1)^k\binm{n}{k}\frac{\binm{y+k}{k}}{\binm{y-n+k}{k}}
\frac{(y-1)y}{(y+k-1)(y+k)}H_k\\
&&\qdn\qqdn\xxqdn\,\,=\:\frac{(-1)^n(1-y+ny)}{n(n-1)(1-y)}\frac{\binm{-y+n}{n}}{\binm{y}{n}}-
\frac{(-1)^n}{n(n-1)}\frac{1}{\binm{y}{n}}.
 \enm
The difference of \eqref{harmonic-c} and the last equation produces
Theorem \ref{thm-h}.
\end{proof}

Setting $x=p$, $y=q$ in  Theorem \ref{thm-h} with
$p,q\in\mathbb{N}_0$ and utilizing \eqref{saalschutz-e}, we obtain
the summation formula involving harmonic numbers.

\begin{corl} \label{corl-h}
Let $p$ and $q$ be both nonnegative integers provided that $q\geq
n$. Then
  \bnm
&&\qdn\xqdn\sum_{k=0}^n(-1)^k\binm{n}{k}\frac{\binm{q+k}{k}}{\binm{q-n+k}{k}}
\frac{(q-1)q}{(q+k-1)(q+k)}H_{p+k}\\
&&\qdn\xqdn\:\,=\frac{(-1)^n(1+p-q+nq)}{n(n-1)(1+p-q)}\frac{\binm{p-q+n}{n}}{\binm{p+n}{n}\binm{q}{n}}
 -\frac{(-1)^n}{n(n-1)}\frac{1}{\binm{q}{n}}.
 \enm
\end{corl}

Applying the derivative operator $\mathcal{D}_x$ to both sides of
Theorem \ref{thm-g}, we get the summation formula involving
generalized harmonic numbers of 3-order.

\begin{thm} \label{thm-i}
Let $x$ and $y$ be both complex numbers. Then
 \bnm
&&\xqdn\sum_{k=0}^n(-1)^k\binm{n}{k}\frac{\binm{y+k}{k}}{\binm{y-n+k}{k}}
\frac{(y-1)y}{(y+k-1)(y+k)}H_{k}^{\langle3\rangle}(x)\\
&&\xqdn\:\:=\,\frac{(-1)^n(1+x-y+ny)}{2n(n-1)(1+x-y)}\frac{\binm{x-y+n}{n}}{\binm{x+n}{n}\binm{y}{n}}
\{A_n(x,y)+B_n(x,y)\},
 \enm
where the two symbols on the right hand side stand for
  \bnm
 &&\xqdn\qdn
 A_n(x,y)=\big[H_{n}^{\langle2\rangle}(x)-H_{n}^{\langle2\rangle}(x-y)\big]+\frac{2ny}{(1+x-y)^2(1+x-y+ny)},\\
&&\xqdn\qdn
B_n(x,y)=\big[H_{n}(x)-H_{n}(x-y)\big]\bigg[H_{n}(x)-H_{n}(x-y)+\frac{2ny}{(1+x-y)(1+x-y+ny)}\bigg].
\enm
\end{thm}

Choosing $x=p$, $y=q$ in  Theorem \ref{thm-i} with
$p,q\in\mathbb{N}_0$ and exploiting \eqref{saalschutz-e}, we gain
the summation formula involving harmonic numbers of 3-order.

\begin{corl} \label{corl-i}
Let $p$ and $q$ be both nonnegative integers satisfying $p\geq q\geq
n$. Then
 \bnm
&&\xqdn\qdn\sum_{k=0}^n(-1)^k\binm{n}{k}\frac{\binm{q+k}{k}}{\binm{q-n+k}{k}}
\frac{(q-1)q}{(q+k-1)(q+k)}H_{p+k}^{\langle3\rangle}\\
&&\xqdn\qdn\:\,=\frac{(-1)^n(1+p-q+nq)}{2n(n-1)(1+p-q)}\frac{\binm{p-q+n}{n}}{\binm{p+n}{n}\binm{q}{n}}
\{C_n(x,y)+D_n(x,y)\},
 \enm
where the corresponding expressions are
 \bnm
 &&\xxqdn
 C_n(p,q)=\big[H_{p+n}^{\langle2\rangle}-H_{p-q+n}^{\langle2\rangle}
 -H_{p}^{\langle2\rangle}+H_{p-q}^{\langle2\rangle}\big]+\frac{2nq}{(1+p-q)^2(1+p-q+nq)},\\
&&\xxqdn
D_n(p,q)=\big[H_{p+n}-H_{p-q+n}-H_{p}+H_{p-q}\big]\\
 &&\quad\times\:\bigg[H_{p+n}-H_{p-q+n}-H_{p}+H_{p-q}+\frac{2nq}{(1+p-q)(1+p-q+nq)}\bigg].
 \enm
\end{corl}

Applying the derivative operator $\mathcal{D}_x$ to both sides of
Theorem \ref{thm-i}, we achieve the summation formula involving
generalized harmonic numbers of 4-order.

\begin{thm} \label{thm-j}
Let $x$ and $y$ be both complex numbers. Then
 \bnm
&&\:\:\sum_{k=0}^n(-1)^k\binm{n}{k}\frac{\binm{y+k}{k}}{\binm{y-n+k}{k}}
\frac{(y-1)y}{(y+k-1)(y+k)}H_{k}^{\langle4\rangle}(x)\\
&&\:\:\:\,=\frac{(-1)^n}{6n(n-1)(1+x-y)}\frac{\binm{x-y+n}{n}}{\binm{x+n}{n}\binm{y}{n}}\\
 &&\:\:\:\,\times\:\bigg\{(1+x-y+ny)E_n(x,y)+\frac{3ny}{1+x-y}F_n(x,y)+G_n(x,y)\bigg\},
 \enm
where the three symbols on the right hand side stand for
  \bnm
 &&\xqdn
 E_n(x,y)=\big[H_{n}(x)-H_{n}(x-y)\big]^3+2\big[H_{n}^{\langle3\rangle}(x)-H_{n}^{\langle3\rangle}(x-y)\big]\\
 &&\qquad\:\:+\:\,
 3\big[H_{n}(x)-H_{n}(x-y)\big]\big[H_{n}^{\langle2\rangle}(x)-H_{n}^{\langle2\rangle}(x-y)\big],\\
&&\xqdn
F_n(x,y)=\big[H_{n}(x)-H_{n}(x-y)\big]^2+\big[H_{n}^{\langle2\rangle}(x)-H_{n}^{\langle2\rangle}(x-y)\big],\\
&&\xqdn
G_n(x,y)=\frac{6ny}{(1+x-y)^2}\big[H_{n}(x)-H_{n}(x-y)\big]+\frac{6ny}{(1+x-y)^3}.\\
\enm
\end{thm}

Selecting $x=p$, $y=q$ in  Theorem \ref{thm-j} with
$p,q\in\mathbb{N}_0$ and availing \eqref{saalschutz-e}, we attain
the summation formula involving harmonic numbers of 4-order.

\begin{corl} \label{corl-j}
Let $p$ and $q$ be both nonnegative integers provided that $p\geq
q\geq n$. Then
 \bnm
&&\qdn\xqdn\sum_{k=0}^n(-1)^k\binm{n}{k}\frac{\binm{q+k}{k}}{\binm{q-n+k}{k}}
\frac{(q-1)q}{(q+k-1)(q+k)}H_{p+k}^{\langle4\rangle}\\
&&\qdn\xqdn\:\,=\frac{(-1)^n}{6n(n-1)(1+p-q)}\frac{\binm{p-q+n}{n}}{\binm{p+n}{n}\binm{q}{n}}\\
 &&\qdn\xqdn\:\,\times\:\bigg\{(1+p-q+nq)U_n(p,q)+\frac{3nq}{1+p-q}V_n(p,q)+W_n(p,q)\bigg\},
 \enm
where the corresponding expressions are
 \bnm
 &&\xxqdn
 U_n(p,q)=\big[H_{p+n}-H_{p-q+n}-H_{p}+H_{p-q}\big]^3
 +2\big[H_{p+n}^{\langle3\rangle}-H_{p-q+n}^{\langle3\rangle}-H_{p}^{\langle3\rangle}+H_{p-q}^{\langle3\rangle}\big]
 \\&&\quad\!+\:\,3\big[H_{p+n}-H_{p-q+n}-H_{p}+H_{p-q}\big]
 \big[H_{p+n}^{\langle2\rangle}-H_{p-q+n}^{\langle2\rangle}-H_{p}^{\langle2\rangle}+H_{p-q}^{\langle2\rangle}\big],\\
&&\xxqdn
V_n(p,q)=\big[H_{p+n}-H_{p-q+n}-H_{p}+H_{p-q}\big]^2
+\big[H_{p+n}^{\langle2\rangle}-H_{p-q+n}^{\langle2\rangle}-H_{p}^{\langle2\rangle}+H_{p-q}^{\langle2\rangle}\big],\\
 &&\xxqdn
W_n(p,q)=\frac{6nq}{(1+p-q)^2}\big[H_{p+n}-H_{p-q+n}-H_{p}+H_{p-q}\big]+\frac{6nq}{(1+p-q)^3}.
 \enm
\end{corl}

 Closed expressions for the following series
 \bnm
 &&\sum_{k=0}^{n}(-1)^k\binm{n}{k}\frac{\binm{y+k}{k}}{\binm{y-n+k}{k}}
 \frac{\binm{y}{t}}{\binm{y+k}{t}}H_{k}^{\langle\ell\rangle}(x)
 \enm
with $t\geq2$ and $\ell\geq5$ can also be given in the same way. The
corresponding conclusions will not be laid out in the paper.

 \textbf{Acknowledgments}

 The work is supported by the National Natural Science Foundation of China (No. 11301120).



\end{document}